\documentclass{amsart}
\usepackage{
 amsmath, 
 amsxtra, 
 amsthm, 
 amssymb, 
 booktabs,
 comment,
 longtable,
 mathrsfs, 
 mathtools, 
 multirow,
 stmaryrd,
 tikz-cd, 
 bbm,
 xr,
 color,
 xcolor,
 enumitem}

 \usepackage[normalem]{ulem}

 \usepackage{colonequals} 
 \usepackage[bbgreekl]{mathbbol}
\usepackage[all]{xy}
\usepackage[nobiblatex]{xurl}
\usepackage[colorlinks=true, linkcolor=magenta, urlcolor=blue, citecolor=blue]{hyperref}
\usepackage{geometry}
\usepackage{mathrsfs}
\newtheorem{theorem}{Theorem}[section]
\newtheorem{lemma}[theorem]{Lemma}

\newtheorem{proposition}[theorem]{Proposition}

\newtheorem{defn}[theorem]{Definition}
\newtheorem{example}[theorem]{Example}

\newcommand\robout{\bgroup\markoverwith {\textcolor{blue}{\rule[0.5ex]{2pt}{0.4pt}}}\ULon}

\newtheorem{lthm}{Theorem} 

\theoremstyle{remark}
\newtheorem{remark}[theorem]{Remark}

\makeatletter
\newcommand{\mylabel}[2]{#2\def\@currentlabel{#2}\label{#1}}
\makeatother

\newcommand{\Gal}{\mathrm{Gal}}
\newcommand{\BSymb}{\mathrm{BSymb}}

\newcommand{\Hom}{\mathrm{Hom}}
\newcommand{\Symb}{\mathrm{Symb}}
\newcommand{\cG}{\mathcal{G}}
\newcommand{\SL}{\mathrm{SL}}

\newcommand{\cor}{\mathrm{cor}}
\newcommand{\ord}{\mathrm{ord}}
\newcommand{\ZZ}{\mathbb{Z}}
\newcommand{\CC}{\mathbb{C}}

\newcommand{\QQ}{\mathbb{Q}}
\newcommand{\Qp}{\mathbb{Q}_p}
\newcommand{\Fp}{\mathbb{F}_p}
\newcommand{\Zp}{\ZZ_p}

\newcommand{\TT}{\mathbb{T}}

\newcommand{\mm}{\mathfrak{m}}
\newcommand{\cO}{\mathcal{O}}
\usepackage[OT2,T1]{fontenc}
\DeclareSymbolFont{cyrletters}{OT2}{wncyr}{m}{n}
\DeclareMathSymbol{\Sha}{\mathalpha}{cyrletters}{"58}
\DeclareMathSymbol\dDelta  \mathord{bbold}{"01}
\definecolor{Green}{rgb}{0.0, 0.5, 0.0}

\usepackage[utf8]{inputenc}
\numberwithin{equation}{section}
\author{Antonio Lei}
\address{Antonio Lei\newline Department of Mathematics and Statistics\\University of Ottawa\\
150 Louis-Pasteur Pvt\\
Ottawa, ON\\
Canada K1N 6N5}
\email{antonio.lei@uottawa.ca}

\author{Robert Pollack}
\address{Robert Pollack\newline Department of Mathematics\\The University of Arizona\\617 N. Santa Rita Ave. \\
Tucson\\ AZ 85721-0089\\USA}
\email{rpollack@arizona.edu}

\author{Naman Pratap}
\address{Naman Pratap\newline Centre for Mathematical Sciences\\ University of Cambridge\\
Wilberforce Road,
Cambridge \\ CB3 0WA,
United Kingdom}
\email{np637@cam.ac.uk}


\begin{document}

\begin{abstract}
Let $E$ be an elliptic curve with good ordinary reduction at an odd prime $p$. Assuming that Greenberg's $\mu=0$ conjecture holds, we show that the $\lambda$-invariants of the Mazur--Tate elements attached to $E$ either stabilise to the $\lambda$-invariant of the $p$-adic $L$-function or they attain the largest possible value at all finite levels. We characterise the latter phenomenon:\ it occurs if and only if  $\ord_p\left(\frac{L(E',1)}{\Omega_{E'}}\right)$ is negative for some $E'$ that is isogenous to $E$.  Furthermore, we relate this condition to congruences with boundary symbols coming from Eisenstein series. We also study the extension of these results to Hecke eigenforms of weight two.
\end{abstract}
\title[]{On the maximality of the $\lambda$-invariants of Mazur--Tate elements}

\maketitle
\section{Introduction}\label{sec: intro2}
Let $E$ be an elliptic curve defined over $\QQ$ and let $f_E$ be the weight two cusp form of level $N_E$ attached to $E$. In this article, we are interested in the modular elements of Mazur and Tate attached to $E$ introduced in \cite{MT}, which we call Mazur--Tate elements in this article. These are defined as elements of group rings of the form $\QQ[\Gal(M/\QQ)]$, where $M$ is an abelian extension of $\QQ$. Let $p$ be an odd prime. We shall focus on the $p$-adic Mazur--Tate elements $\theta_n(E)$, which interpolate the twisted $L$-values $L(E,\chi,1)$ for Dirichlet characters $\chi$ of $p$-power conductor. 

Let $k_n$ denote the unique sub-extension of $\QQ(\mu_{p^\infty})/\QQ$ that is of degree $p^n$. We shall regard the Mazur--Tate element $\theta_n(E)$ for the extension $k_n/\QQ$ as an element of $\Qp[\Gal(k_n/\QQ)]$. One can define $\mu$- and $\lambda$-invariants of elements of $\Qp[\Gal(k_n/\QQ)]$ as one does for elements of the Iwasawa algebra (see Definition~\ref{def:inv}).
The $\lambda$-invariant of $\theta_n(E)$ can be arbitrarily large, as illustrated by the following example.

\begin{example}\label{example: X0(11)}
    Consider $E=X_0(11)$ and $p=5$. For $n \geq 0$,
$$
\mu(\theta_n(E))=0 \quad \text{and} \quad \lambda(\theta_n(E))=p^n-1. 
$$
This behaviour can be explained by the existence of a congruence between the modular symbol of $E$ with a boundary symbol (a modular symbol corresponding to an Eisenstein series, see \S\ref{sec:boundary}; in particular, Remark~\ref{rk:boundary}) modulo $p$, which leads to the above expression for the $\lambda$-invariants. 
\end{example}
Note that the $\lambda$-invariant of a non-zero element in $\Qp[\Gal(k_n/\QQ)]$ is at most $p^n-1$. We shall refer to the property $\lambda(\theta_n(E))=p^n-1$ as the \emph{maximality} of the $\lambda$-invariants of $\theta_n(E)$. Note that this phenomenon is expected to occur only when the $\Gal(\overline{\QQ}/\QQ)$-module $E[p]$ is reducible, as when $E[p]$ is irreducible, we have:
\begin{theorem}\label{thm: PW11 irreducible thm}
Let $E/\QQ$ be an elliptic curve with good ordinary reduction at $p$ such that $E[p]$ is irreducible as a $\Gal(\overline{\QQ}/\QQ)$-module. Let $L_p(E)$ be the $p$-adic $L$-function of Mazur--Swinnerton-Dyer attached to $E$ constructed in \cite{MSD74}. If $\mu(L_p(E)) = 0$, then, for $n$ large enough
$$
\mu(\theta_n(E)) = 0 \quad \text{and} \quad \lambda(\theta_n(E)) = \lambda(L_p(E)). 
$$
\end{theorem}

\begin{proof}
See \cite[Proposition 3.7]{PW}.
\end{proof}
Greenberg's $\mu=0$ conjecture predicts that $\mu(L_p(E)) = 0$ when $E[p]$ is an irreducible $\Gal(\overline{\QQ}/\QQ)$-module. In other words, we expect that $\lambda(\theta_n(E))$ should always be bounded in this case. See also Theorem~\ref{thm: mt invariants are Lp invariants} for a slight generalization of Theorem~\ref{thm: PW11 irreducible thm}, where we relax the hypothesis on the irreducibility of $E[p]$.

One of the main goals of this article is to characterise the Iwasawa invariants of $\theta_n(E)$ at good ordinary primes $p$ for $n\gg0$ when $E[p]$ is reducible. In fact, we show that the previous two kinds of behaviour, those exhibited in Example \ref{example: X0(11)} and Theorem \ref{thm: PW11 irreducible thm}, are the only possibilities for the Iwasawa invariants of $\theta_n(E)$. Furthermore, our result characterises exactly when these two behaviours happen in terms of the $p$-adic valuation of the normalised $L$-value at $s=1$.

\subsection{Main results}
In this section, we describe the main results of this article. First, we show that the maximality of the Iwasawa invariants of the Mazur--Tate elements observed in Example \ref{example: X0(11)} can be generalised to elliptic curves $E$ satisfying the condition $\ord_p(L(E,1)/\Omega_E)<0$, where $\Omega_E$ is the real N\'eron period of $E$.
\begin{lthm}[Theorems~\ref{thm: Lvaldenom-version2} and \ref{thm: converse}]\label{thmA}
Let $E$ be an elliptic curve defined over $\QQ$ with good ordinary reduction at $p$. 
\begin{itemize}[leftmargin=0.5cm] 
    \item[(1)] If $\ord_p(L(E,1)/\Omega_{E})<0$, then
    \[\mu(\theta_n(E))=\ord_p(L(E,1)/\Omega_{E})\quad\text{and }\quad\lambda(\theta_n(E))=p^n-1\]
    for all $n\geq 0$.
    \item[(2)] If $\mu(L_p(E))=0$ and $\lambda(\theta_n(E))=p^n-1$
    for $n\gg0$, then
    \[\ord_p\left(\frac{L(E,1)}{\Omega_E}\right)<0.\]
\end{itemize}
\end{lthm}
\begin{remark}
    Greenberg's $\mu=0$ conjecture \cite[Conjecture 1.11]{Greenberg} predicts that there exists at least one curve $E'$ in the isogeny class of $E$ satisfying $\mu(L_p(E'))=0$. Furthermore, $\lambda(\theta_n(E))$ is constant within an isogeny class.
    Thus, admitting Greenberg's conjecture, Theorem \ref{thmA} implies that the $\lambda$-invariants of $\theta_n(E)$ are maximal if and only if $E$ is isogenous to a curve $E'$ such that $\ord_p(L(E',1)/\Omega_{E'})<0$. 
\end{remark}

    The proof of Theorem \ref{thmA} relies crucially on a result of Wuthrich \cite{wuthrichint} on the integrality of the $p$-adic $L$-function of $E$. More precisely, it asserts that $L_p(E)$ is an element of the Iwasawa algebra $\Lambda$ when $E$ is good ordinary at $p$ (see Theorem~\ref{thm: wuthrich Lp integrality}). The condition $\ord_p(L(E,1)/\Omega_E)<0$ implies that $\theta_0(E)$ is non-integral. Combining these facts, we deduce that $\theta_n(E)$ is non-integral for all $n$. 
    The exact formula for $\lambda(\theta_n(E))$ given in part (a) is obtained by analyzing the effect on $\lambda$-invariants under the natural trace/corestriction map $\Qp[\Gal(k_n/\QQ)]\to\Qp[\Gal(k_{n+1}/\QQ)]$ (see Lemma \ref{lem: corlambda} for details). For the converse implication (part (b) of the theorem), we exploit the norm relations satisfied by $\theta_n(E)$ to show that the maximality of $\lambda$-invariants and the vanishing of $\mu(L_p(E))$ imply that $\theta_n(E)$ is non-integral for all $n$.

Assuming the Birch and Swinnerton-Dyer conjecture for $E/\QQ$, the condition $\ord_p(L(E,1)/\Omega_E)<0$ implies the reducibility of $E[p]$ as a $\Gal(\overline{\QQ}/\QQ)$-module as $|E_{\rm tor}(\QQ)|^2$ is the only term appearing in the denominator of the BSD quotient. However, the converse is not true (see Example~\ref{example: 174b}). 
Our next result shows that the condition $\ord_p\bigl(L(E,1)/\Omega_E\bigr) < 0$ induces a dichotomy in the behavior of Mazur-Tate elements:\ either their $\mu$- and $\lambda$-invariants stabilise to the corresponding invariants of the $p$-adic $L$-function or their $\lambda$-invariants are maximal.

\begin{lthm}[Theorem~\ref{thm: classify}]\label{thmB}
     Let $E$ be an elliptic curve over $\QQ$ with good ordinary reduction at an odd prime $p$. Assume $E$ is isogenous over $\QQ$ to an elliptic curve $E'/\QQ$ with $\mu(L_p(E'))=0$.  Then, we have the following dichotomy:
\begin{itemize}
        \item[(a)] $\theta_n(E') \in \Lambda_n$ for all $n \geq 0$.  In this case, $$\mu(\theta_n(E))=\mu(L_p(E)) \quad \text{and} \quad \lambda(\theta_n(E))=\lambda(L_p(E))
        $$
        for $n$ sufficiently large, or
        \item[(b)] $\theta_n(E') \notin \Lambda_n$ for all $n \geq 0$.  In this case,
        $$
        \mu(\theta_n(E))=\ord_p(L(E,1)/\Omega_E) \quad \text{and} \quad \lambda(\theta_n(E))=p^n-1
        $$
        for all $n \geq 0$.
    \end{itemize}
    Furthermore, case (b) occurs if and only if $\ord_p(L(E',1)/\Omega_{E'})<0$.
\end{lthm}
We remind the reader that Greenberg conjectures that our assumption on the $\mu$-invariant always holds. 

Our methods can be easily extended to obtain analogues of Theorems \ref{thmA} and \ref{thmB} for $p$-ordinary Hecke eigenforms of weight $2$ and level coprime to $p$ (see Theorems \ref{thm: ThA for mod forms} and \ref{thm: ThB for mod forms}). 
For higher weight $p$-ordinary Hecke eigenforms, if their modular symbols are normalised using the cohomological periods (see Remark~\ref{rem: periods}), then under the assumption that the $\mu$-invariant of the associated $p$-adic $L$-function vanishes, it can be shown that the first case of Theorem \ref{thmB} always holds (see \cite[proof of Proposition 3.7]{PW}).

In the present article, we concentrate on good ordinary primes. For multiplicative primes, $\theta_{n}(E)$ have the same $\lambda$-invariants as the $p$-adic $L$-function for $n\gg 0$. When $p$ is a supersingular prime, we also have a good understanding of the growth of $\lambda(\theta_{n}(E))$; see \cite[Theorem 4.1]{PW}. In the case of additive primes, different patterns can emerge. This has been studied in \cite{LLP-additive} and \cite{doyon-lei}. Interestingly, the maximality of $\lambda$-invariants can occur for additive primes; see Example~\ref{example: 50b}. 

\subsection{Interpreting our results in terms of Eisenstein congruences}
As mentioned above, the Birch and Swinnerton-Dyer conjecture predicts that $\ord_p(L(E,1)/\Omega_E)<0$ implies that $E(\QQ)[p]$ is non-trivial.  In this case, when $E(\QQ)[p] \neq 0$, the semi-simplification of the $\Gal(\overline{\QQ}/\QQ)$-representation $E[p]$ is isomorphic to $1\oplus \omega$, where $\omega$ denotes the mod $p$ cyclotomic character. Thus, the associated cusp form $f_E=\sum_{n} a_n(E)q^n$ satisfies 
 \[a_\ell(E)\equiv \ell+1 \mod{p}\]
 for all primes $\ell\nmid N_E$. Furthermore, $a_\ell(E)=-\epsilon_\ell\in\{1,-1\}$ for primes $\ell\mid \mid N_E$, where $\epsilon_\ell$ is the eigenvalue of the Atkin--Lehner operator $W_\ell$ acting on $f_E$.
    Equivalently, since $f_E$ is a newform of level $N_E$, one may interpret $a_\ell(E)$ as the eigenvalue of the $U_\ell$-operator when $\ell\mid \mid N_E$ since $U_\ell=-W_\ell$ on new forms.
    
Note that there exists an Eisenstein series $E_{2,N_E}=\sum a_n(E_{2,N_E})q^n$ on $\Gamma_0(N_E)$ of weight 2 such that  
$a_n(E_{2,N_E})\equiv a_n(E)\mod{p}$ for all $n$ when $E$ has non-trivial $p$-torsion over $\QQ$ and $p\nmid N_E$. As mentioned earlier, Eisenstein series give rise to boundary symbols. One might hope to deduce from the aforementioned congruence to a congruence relation between the modular symbols attached to $f_E$ and the boundary symbol defined using $E_{2,N_E}$.

The space of weight 2 modular symbols of level $N$ can be canonically identified with the compactly supported cohomology group $H^1_c(Y_0(N),\CC)$. Thus, in order to pass from a congruence of Hecke eigenvalues to a congruence of the respective modular symbols (equivalently, a congruence of the corresponding cohomology classes), one usually requires a mod $p$ `multiplicity one' result for a cohomology group. Such a result would imply that a modular symbol is uniquely determined (up to scaling) by the Hecke eigenvalues.
Multiplicity one results of this kind are well established in the residually irreducible case, see, for example, \cite[Theorem 9.2]{edixhoven}. In the residually reducible case, the situation becomes more subtle, and one usually has to study Eisenstein completions of the Hecke algebra and the \textit{Eisenstein ideal}, as illustrated by the following example.
\begin{example}[$X_0(11)$ at $p=5$, revisited]\label{ex: meeting X0(11) again}
Let $\TT_{N}$ denote the Hecke algebra acting on modular forms of weight 2 and prime level $N$ generated by the usual Hecke operators $T_\ell$ for $\ell\nmid N$ and $U_N$. Let $\TT_N^{\mathrm{Eis}}$ denote the completion of $\TT_N$ at the ideal $(T_\ell-(\ell+1), U_N-1,p)$.

In his celebrated work \cite{Mazur77}, Mazur studied various properties of the Hecke algebra $\TT_{N}^{\mathrm{Eis}}$. He showed that the cuspidal quotient $\TT_{N}^{\mathrm{Eis},0}$ of $\TT_{N}^{\mathrm{Eis}}$ is a Gorenstein ring under certain explicit conditions on $(N,p)$. In particular, his results apply to $(N,p)=(11,5)$. Let $\mm$ denote the maximal ideal of $\TT_{N}^{\mathrm{Eis},0}$ (this corresponds to the residual representation $1\oplus \omega$). The Gorenstein property allows one to show
$\dim_{\Fp}J_0(N)(\overline{\QQ}_p)[\mm]=2$, which implies, using \cite[(7.5)]{Mazur77},
\[\dim_{\Fp} H^1(X_0(N),\Fp)^\pm[\mm]=1,\]
i.e., the space of cuspidal modular symbols with system of mod $p$ Hecke eigenvalues corresponding to $\mm$ is one dimensional.
In this case, $\TT_{Np}^{\mathrm{Eis}}$ is in fact Gorenstein, which can be used to show that the \textit{full} space of modular symbols satisfies
\[\dim_{\Fp} H^1(Y_0(N),\Fp)^\pm[\mm]=1.\]
We refer to this property as mod $p$ multiplicity one for level $N$.
Thus, up to a non-zero constant, the mod $5$ reduction of $\phi_E$ equals the mod 5 reduction of a boundary symbol, since they are both $\mathbb{F}_p$-valued modular symbols that are annihilated by $\mm$. 
\end{example}

\begin{remark}
    In \cite{doyon-lei2}, the $\lambda$-invariants for the Mazur--Tate elements attached to the weight 12 modular form $\Delta$ arising from the Ramanujan tau function were studied for the primes $p\in\{3,5,7\}$. These primes are Eisenstein, i.e., $\Delta$ satisfies a congruence modulo $p$ with certain Eisenstein series in weight $2$. It is shown in \textit{loc.\ cit.} that $\lambda(\theta_n(\Delta))$ is maximal (i.e., equal to $p^n-1$) for $p=5,7$ for $n\geq 1$ by verifying computationally that the modular symbols of $\Delta$ are congruent to a weight 2 boundary symbol modulo $p$. In this setting, mod $p$ multiplicity one results were not available to establish these congruences theoretically. This motivates our study of the maximality phenomenon for elliptic curves in the present article. 
\end{remark}
In Theorem \ref{thmA}, we deduce the maximality of $\lambda(\theta_n)$ without invoking a congruence with a boundary symbol (thus avoiding questions of mod $p$ multiplicity one). However, our data suggest that the maximality of $\lambda(\theta_n)$ occurs for an elliptic curve of square-free conductor $N$ (and $p$ good ordinary) if and only if mod $p$ multiplicity one holds for level $N$. Using the results of Wake--Wang-Erickson \cite{WAKE2021107543}, which studies the appropriate generalisation of Mazur's results on $\TT_N$ for prime $N$ to squarefree $N$, we are able to verify in each of the examples we have considered that the relevant Hecke algebra is indeed Gorenstein. Furthermore, congruences with boundary symbols are also observed, indicating that a mod $p$ multiplicity one result might actually hold. 
However, proving that the Gorenstein property implies mod $p$ multiplicity one in our current setting has eluded us so far.

\subsection{Organisation}
In \S \ref{sec: modsymbs intro}, we review the definitions of modular symbols and Mazur--Tate elements.
In \S \ref{sec: form2}, we prove preliminary results that will be used to prove our main theorems. We prove Theorems \ref{thmA} and \ref{thmB} in \S\ref{S:proofA} and \ref{S:proofB}, respectively. Both Theorems \ref{thmA} and \ref{thmB} are followed by the appropriate extension of our results to modular forms in \ref{S: proof of main res}. In \S\ref{sec: remarks, examples}, we discuss partial progress in relating Theorem~\ref{thmA} to congruences with boundary symbols and mod $p$ multiplicity one for modular symbols. We also give explicit examples related to our results, one of which illustrates an interesting utility of our results in detecting when mod $p$ multiplicity one fails.

\subsection{Acknowledgement}

Parts of this article are contained in NP's master's thesis at IISER, Pune (available \href{http://dr.iiserpune.ac.in:8080/xmlui/handle/123456789/10021}{here}). The authors thank Anthony Doyon, Rylan Gajek-Leonard and Preston Wake for interesting discussions related to the content of this article.  RP’s research has been partially supported by NSF
grant DMS-2302285 and by Simons Foundation Travel Support Grant for Mathematicians MPS-
TSM-00002405.

\section{Modular symbols, boundary symbols and Mazur--Tate elements}\label{sec: modsymbs intro}
\subsection{Modular symbols}
Let $R$ be any commutative ring. Let $\dDelta$ denote the abelian group of divisors on $\mathbb{P}^1(\QQ)$, and let $\dDelta^0$ denote the subgroup of degree 0 divisors. Let $\SL_2(\ZZ)$ act on $\dDelta^0$ by linear fractional transformations, which allows us to endow $\Hom(\dDelta^0, R)$ with a right action of $\SL_2(\ZZ)$ via
$$(\varphi \mid_{\gamma})(D) = (\varphi(\gamma \cdot D))\mid_{\gamma},$$
where $\varphi \in \Hom(\dDelta^0, V_{g}(R))$, $\gamma \in \SL_2(\ZZ)$ and $D \in \dDelta^0$.
\begin{defn}\label{defn:modsymb}
    Let $\Gamma\leq \SL_2(\ZZ)$ be a congruence subgroup. We define $\Hom_{\Gamma}(\dDelta^0, R)$ to be the space of $R$-valued \textbf{modular symbols} (of weight two), level $\Gamma$ for some commutative ring $R$, and we denote this space by $\Symb(\Gamma,R)$.
\end{defn}
\begin{remark}
    There is a canonical isomorphism 
    \[\text{Symb}(\Gamma, R) \cong H^1_c(\Gamma, R)\] where $H^1_c(\Gamma, {R})$ denotes the compactly supported cohomology group (see \cite[Proposition~4.2]{ash-ste}). 
\end{remark}

For $f \in S_2(\Gamma)$, we define the \textbf{modular symbol associated with $f$} as 
\[\xi_f: \{s\}-\{r\} \to 2\pi i \int_s^r f(z)dz,\] 
which is an element of $\Symb(\Gamma, \CC)$ as $f$ is a holomorphic cusp form. Let $A_f$ be the field of Fourier coefficients of $f$ and fix a prime $p$. The matrix $\iota \colonequals \begin{psmallmatrix}
    -1& 0 \\ 0 & 1
\end{psmallmatrix}$ acts as an involution on $\Symb(\Gamma, \CC)$ and we decompose $\xi_f=\xi_f^+ + \xi_f^-$ with $\xi_f^\pm$ in the $\pm1$-eigenspace of $\iota$ respectively. By a theorem of Shimura, there exist $\Omega^\pm \in \CC$ such that ${\xi_f^\pm/\Omega^\pm}$ take values in $A_f$, and in $\overline{\QQ}_p$ upon fixing an embedding of $\overline{\QQ}\hookrightarrow \overline{\QQ}_p$ (which we fix for the rest of the article). Define $\phi_f^\pm \colonequals \psi_f^\pm/\Omega^\pm$, and $\phi_f \colonequals \phi_f^+ + \phi_f^-$, which we regard as an element of $\Symb(\Gamma, \overline{\QQ}_p)$. Let $K_f$ denote the completion of the image of $A_f$ in $\overline{\QQ}_p$ and let $\mathcal{O}_f$ denote the ring of integers of $K_f$. We can choose $\Omega^+$ and $\Omega^-$ so that each of $\phi_f^+$ and $\phi_f^-$ takes values in $\mathcal{O}_f$ and that each takes on at least one value in $\mathcal{O}_f^\times$. We denote these periods $\Omega_f^\pm$; they are called \textbf{cohomological periods} of $f$, which are well-defined up to $p$-adic units (see \cite[Definition 2.1]{PW}).

For an elliptic curve $E$ defined over $\QQ$, we are supplied with the real and imaginary \textbf{Néron periods}, which we denote by $\Omega_E^+$ and $\Omega_E^-$, respectively. We shall use $\Omega_E$ to denote the real Néron period $\Omega_E^+$ in our arguments. 
Let $\phi_{E}$ denote the modular symbol attached to $E$ normalised by $\Omega_E^\pm$, so that
\[\phi_{E} \colonequals \frac{\xi^+_{f_E}}{\Omega^+_{E}}+  \frac{\xi^-_{f_E}}{\Omega^-_{E}}.\] 
It takes values in $\Qp$ but \textit{a priori} is not guaranteed to take values in $\Zp$.
\subsection{Mazur--Tate elements and $p$-adic $L$-functions}\label{ssec: MT and Lp}
For a non-negative integer $n$, let $\mathcal{G}_n \colonequals \Gal(\QQ(\mu_{p^n})/\QQ)$. Recall from the introduction that $k_n$ denotes the unique subextension of $\QQ(\mu_{p^\infty})/\QQ$ that is of degree $p^n$. We write $G_n=\Gal(k_n/\QQ)$.
For $a \in (\ZZ/p^n\ZZ)^\times$, we write $\sigma_a\in\cG_n$ for the element that satisfies $\sigma_a(\zeta)=\zeta^a$ for $\zeta \in \mu_{p^n}$.
\begin{defn}\label{defn: MT elements}
    For a modular symbol $\varphi \in \Symb(\Gamma, R)$, define the associated Mazur--Tate element of level $n\geq 1$ by 
\[\vartheta_n(\varphi)= \sum_{a \in (\ZZ/p^n\ZZ)^\times}\varphi(\{\infty\}-\{a/p^n\})\cdot \sigma_a \in R[\mathcal{G}_n].\]
    
    Let $f\in S_2(\Gamma)$. We define the \emph{the $p$-adic Mazur--Tate element} of level $n\ge0$ associated with $f$ to be the image of $\vartheta_{n+1}(\phi_f)$ under the natural projection $\overline{\QQ}_p[\cG_{n+1}]\to\overline{\QQ}_p[G_n]$. We denote this element by  $\theta_{n}(f)$.
\end{defn}

For an elliptic curve $E/\QQ$, we normalise $\theta_{n}(E)\colonequals\theta_{n}(\phi_{E}) \in \Qp[G_n]$ using the N\'eron periods. We define $\theta_{n}(\phi_{E,\mathrm{Coh}})$ to be the corresponding elements under the normalisation by the cohomological periods of $f_E$.
\begin{remark}\label{rem: theta0}
    Suppose $E$ has good reduction at a prime $p>2$ and let $f_E\in S_2(\Gamma_0(N_E))$ denote the associated cusp form. As $\theta_0(E)$ is the image of $\vartheta_1(\phi_f)$ under the natural map $\overline{\QQ}_p[\cG_1] \to \overline{\QQ}_p[G_0]$, we have $\theta_0(E) =\displaystyle \sum_{a\in (\ZZ/p\ZZ)^\times} \phi_E(\{\infty\}-\{a/p\})\sigma_1$,
    where $\sigma_1$ is the identity automorphism on $\QQ$.  Since $\phi_E$ is an eigenvector for the Hecke operator $\displaystyle T_p=\begin{psmallmatrix}
        p & 0 \\ 0 &1
    \end{psmallmatrix}+\sum_{u=0}^{p-1}\begin{psmallmatrix}
        1 & u \\ 0 & p
    \end{psmallmatrix}$, we have 
\[a_p(E)\phi_E=\phi_E(\{\infty\}-\left\{0\right\})+\sum_{a=0}^{p-1} \phi_E(\{\infty\}-\left\{a/{p}\right\}),\]
    which implies
    \begin{equation}
        \theta_0(E)= (a_p(E)-2)\phi_E(\{\infty\}-\{0\})\sigma_1= \left((a_p(E)-2)\frac{L(E,1)}{\Omega_E}\right)\sigma_1. \label{eq:interpolation}
    \end{equation}
    We have implicitly used the fact that $\{\infty\}-\{a\}$ is $\Gamma_0(N_E)$-equivalent to $\{\infty\}-\{0\}$ for all $a \in \ZZ/p\ZZ$. The last equality above follows from $\phi_E(\{\infty\}-\{0\})\cdot \Omega_E=\displaystyle 2\pi i\int^\infty_0 f_E(z) dz=L(E,1)$. Similarly, for $f\in S_2(\Gamma_1(N),\epsilon_f)$, we have 
    \[\theta_0(f)= (a_p(f)-\epsilon_f(p)-1)\phi_f(\{\infty\}-\{0\}).\]
\end{remark}

We briefly recall the construction of the $p$-adic $L$-function of a $p$-ordinary Hecke eigenform $f=\sum_{n} a_n(f)e^{2\pi inz}\in S_2(\Gamma_1(N),\epsilon_f)$ when $ p\nmid N$. Let $\alpha$ denote the unique $p$-adic unit root of the Hecke polynomial $X^2-a_p(f)X+\epsilon_f(p)p$. We consider the $p$-stabilisation 
\begin{equation}\label{eq:f alpha defn}
    f_{\alpha}(z)\colonequals f(z)- \frac{\epsilon_f(p)p}{\alpha}f(pz).
\end{equation}
\[\]
Define $\phi_{f}^\alpha$ as the `$p$-stabilised' modular symbol attached to $f$ normalised by some periods $\Omega^\pm$, i.e.,
\begin{equation}\label{eq:stabilization}
    \phi_{f}^\alpha \colonequals\phi_f-\frac{\epsilon_f(p)}{\alpha}\phi_f|\begin{psmallmatrix}
    p & 0 \\ 0 & 1
\end{psmallmatrix}.
\end{equation}
This gives us a norm-compatible system $\displaystyle\left\{\frac{1}{\alpha^{n+1}} \theta_{n}(\phi_{f}^{\alpha})\right\}_n$ since $\phi_{f}^\alpha$ is an eigenvector for the $U_p$-operator with eigenvalue $\alpha$. 

Let $\Lambda_{n,\mathcal{O}_f}$ denote $\mathcal{O}_f[G_n]$ and define $\Lambda_{\mathcal{O}_f}=\varprojlim\Lambda_{n,\mathcal{O}_f}$, where the connecting maps are the natural projections. When $\mathcal{O}_f=\Zp$, we shall write $\Lambda_n$ for $\Lambda_{n,\Zp}$ for simplicity.
Then,
\begin{align*}
   L_p(f)&=\varprojlim_{n}\frac{1}{\alpha^{n+1}} \theta_{n}(\phi_{f}^{\alpha})\in \Lambda_{\mathcal{O}_f}\otimes \overline{\QQ}_p
\end{align*}
is the trivial isotypic component of the $p$-adic $L$-function attached to $f$.  We use the notations $L_p(f)$ and $L_p(f, T)$ interchangeably, where $L_p(f, T)$ denotes the image of $L_p(f)$ under our fixed isomorphism $\Lambda_{\mathcal{O}_f}\otimes \overline{\QQ}_p\cong\mathcal{O}_f[[T]]\otimes \overline{\QQ}_p$. One can also define the $p$-adic $L$-function as an element of $\mathcal{O}_f[[\Gal(\QQ(\mu_{p^\infty})/\QQ]]\otimes \overline{\QQ}_p$ by considering the norm-compatible system built from $\frac{1}{\alpha^{n}}\vartheta_n(\phi_{f}^\alpha)$ directly. We denote this inverse limit by $\mathcal{L}_p(f)$.  

\begin{remark}[\textbf{On periods}]\label{rem: periods}
Our methods in this article rely on the integrality of the associated $p$-adic $L$-function. From the existence of the cohomological periods, it is clear, for a $p$-ordinary Hecke eigenform, that there exist periods that ensure the $p$-adic $L$-function is an element of $\Lambda_{\cO_f}$. We denote such periods by by $\Omega_{f, \bullet}^\pm$. For simplicity, we shall write $\Omega_{f,\bullet}$ for $\Omega_{f,\bullet}^+$. 
 
From now on, whenever we refer to $\phi_f$ for a $p$-ordinary Hecke eigenform $f$, we assume that the $\phi_f$ is normalised by a choice of $\Omega_{f,\bullet}^\pm$. The cohomological periods associated with the ordinary $p$-stabilisation $f_\alpha$ of $f$ (see \eqref{eq:f alpha defn} for the definition of $f_\alpha$) are possible candidates for $\Omega_{f,\bullet}^\pm$. While such a normalisation ensures integrality of the $p$-adic $L$-function, the modular symbol $\phi_f$ takes values in $K_f$, and not necessarily in $\mathcal{O}_f$.
    
\end{remark}

A choice of $\Omega_{f,\bullet}^\pm$ ensures $\mathcal{L}_p(f) \in \mathcal{O}_f[[\Gal(\QQ(\mu_{p^\infty})/\QQ]]$ and $L_p(f)\in \Lambda_{\mathcal{O}_f}$.
For elliptic curves, since we normalise $\phi^\alpha_E$ by the Neron periods $\Omega_E^\pm$, the corresponding $p$-adic $L$-functions, which we denote by $L_p(E)$ and $\mathcal{L}_p(E)$ are \textit{a priori} elements of $\Lambda\otimes \Qp$ and $ \Zp[[\Gal(\QQ(\mu_{p^\infty})/\QQ]]\otimes \Qp$, respectively. Nonetheless, we have:
\begin{theorem}\label{thm: wuthrich Lp integrality}
    Let $E$ be an elliptic curve defined over $\QQ$ with good ordinary reduction at an odd prime $p$. Then the $p$-adic $L$-function $\mathcal{L}_p(E)$ normalised by $\Omega_E^\pm$ lies in $\Zp[[\Gal(\QQ(\zeta_{p^\infty})/\QQ)]]$.
\end{theorem}
\begin{proof}
    This is \cite[Corollary~18]{wuthrichint}.
\end{proof}
\begin{remark}
    As a corollary of the above theorem, we deduce that the trivial isotypic component of $\mathcal{L}_p(E)$, i.e., $L_p(E)$ lies in $\Lambda\cong\Zp[[T]]$.
\end{remark}
Theorem \ref{thm: wuthrich Lp integrality} allows us to deduce the following:
\begin{lemma} \label{mtinteg}
   Let $E$ be an elliptic curve defined over $\QQ$ with good ordinary reduction at a prime $p>2$. 
    Then
 \[{\phi_{E}^\alpha(\{\infty\}-\{a/p^n\})} \in \Zp\] for all $a \in (\ZZ/p^n\ZZ)^\times$ and $n\geq 0$.  
\end{lemma}
\begin{proof}
     Theorem \ref{thm: wuthrich Lp integrality} says that $\mathcal{L}_p(E)$ is an element of \[\Zp[[\Gal(\QQ(\zeta_{p^\infty})/\QQ)]]\cong \underset{n}{\varprojlim}\Zp[\mathcal{G}_n]\cong \Zp[[\Zp^\times]].\]
    The projection of $\mathcal{L}_p(E)$ to $\Zp[\mathcal{G}_n]$ is 
    $\vartheta_n(\phi_E^\alpha)/\alpha^n$, where $\vartheta_n(\phi_E^\alpha)$ is defined as 
    \[\vartheta_n(\phi_E^\alpha) \colonequals \sum_{a \in (\ZZ/p^n\ZZ)^\times} \phi^\alpha_{E}(\{\infty\}-\{a/p^n\})\cdot \sigma_a \]
    for $\sigma_a \in \mathcal{G}_n = \Gal(\QQ(\zeta_{p^n})/\QQ)$. As $\alpha \in \Zp^\times$, we have $ \vartheta_n(\phi_E^\alpha) \in \Zp[\mathcal{G}_n]$,
    implying $\phi^\alpha_E(\{\infty\}-\{a/p^n\}) \in \Zp$ for all $a \in (\ZZ/p^n\ZZ)^\times$ and $n\geq 0$.
\end{proof}
We note that the same proof works for a $p$-ordinary eigenform $f\in S_k(\Gamma_1(N),\epsilon_f)$ with $p\nmid N$ and implies that $\phi_f^\alpha((\{\infty\}-\{a/p^n\}) \in \Zp$ for all $a \in (\ZZ/p^n\ZZ)^\times$ and $n\geq 0$, since we normalise the modular symbols by $\Omega_{f,\bullet}^\pm$.
\section{General results on Iwasawa invariants}\label{sec: form2}

Let $\gamma_n$ be a generator of ${G}_n$ and let $K$ be a finite extension of $\Qp$ with a fixed uniformizer $\varpi$. Let $\cO_K$ denote the ring of integers of $K$.
For any element $F \in K[{G}_n]$, we may write it as a polynomial 
$\sum_{i=0}^{p^n-1}a_iT^i$ with $T=\gamma_n-1$ and $a_i\in K$. 
\begin{defn}[Iwasawa invariants]\label{def:inv}
    The $\mu$ and $\lambda$-invariants of $F=\sum_{i=0}^{p^n-1}a_iT^i \in K[G_n]$ are defined as 
\begin{align*}
    \mu(F) &= \underset{i}{\min}\{\ord_\varpi(a_i)\},\\
    \lambda(F) &= \min\{ i : \ord_\varpi(a_i) = \mu(F)\},
\end{align*}
where $\ord_\varpi$ is the valuation such that $\ord_\varpi(\varpi)=1$. 
\end{defn}
These invariants are independent of the choice of $\gamma_n$.
One can directly define $\mu$ and $\lambda$-invariants for an element of the finite level group algebra $K[G_n]$ without reference to a generator $\gamma_n$; for more details, see \cite[\S~3.1]{PW}.

We begin with a lemma involving the Iwasawa invariants of a sum of elements in $ K[G_n]$, which will be utilised in the proof of Theorem \ref{thmA} (Theorem \ref{thm: converse}).
\begin{lemma}\label{lem:abstract}
    Let $F_1,F_2\in K[G_n]$ with $\lambda(F_1+F_2)<\lambda(F_1),\lambda(F_2)$. Then $$\mu(F_1+F_2)>\mu(F_1)=\mu(F_2) \quad\text{and} \quad\lambda(F_1)=\lambda(F_2).$$
\end{lemma}
\begin{proof}
    Suppose that $\mu(F_1)>\mu(F_2)$. Then $\mu(F_1+F_2)=\mu(F_2)$ and $\lambda(F_1+F_2)=\lambda(F_2)$. This contradicts $\lambda(F_1+F_2)<\lambda(F_2)$. Thus, we deduce $\mu(F_1)=\mu(F_2)$.

    After multiplying by a scalar if necessary, we may assume that $\mu(F_1)=\mu(F_2)=0$. Suppose that $\lambda(F_1)>\lambda(F_2)$. Then we have $\lambda(F_1+F_2)=\lambda(F_2)$. This contradicts $\lambda(F_1+F_2)<\lambda(F_2)$. Hence, we deduce $\lambda(F_1)=\lambda(F_2)$.

    Let $d=\lambda(F_1)=\lambda(F_2)$. There exists $u_i\in \cO^\times$ such that
    \[
    F_i\equiv u_i T^d\mod \varpi\cO[G_n],\quad i=1,2.
    \]
    In particular, $F_1+F_2\equiv (u_1+u_2)T^d\mod \varpi\cO[G_n]$. As $\lambda(F_1+F_2)<d$, this implies that $u_1+u_2\equiv 0\mod \varpi$. Hence, $\mu(F_1+F_2)>0$, as desired.
\end{proof}

Let $\pi_{n-1}^{n} : G_{n} \to  G_{n-1}$ be the natural projection map. For $\sigma \in  G_{n-1}$, define the corestriction map
\[\cor_{n-1}^n(\sigma) \colonequals \sum_{\substack{\pi_{n-1}^{n}(\tau)=\sigma \\ \tau \in  \Gal(k_{n}/\QQ)}} \tau\in K[G_n]\]
   which extends $K$-linearly to  a map $K[G_{n-1}]\to K[  G_n]$. We recall some standard facts about Iwasawa invariants of Mazur--Tate elements under the natural projection and corestriction maps.
\begin{lemma}\label{lem: corlambda}
\begin{itemize}
    \item[(1)] For an integer $N\geq 1$ and a prime $p$, let $\phi \in \Symb(\Gamma_0(N),\cO)$ be a modular symbol. For $n\geq 1$, we have 
\[\theta_{n}(\phi\mid\begin{psmallmatrix}
    p &0 \\ 0 & 1
\end{psmallmatrix})= \cor^n_{n-1}(\theta_{n-1}(\phi)).\]
       \item[(2)] For $\theta \in K[G_{n-1}]$, we have 
    \[\mu(\cor_{n-1}^n(\theta))=\mu(\theta) \text{ and } \lambda(\cor_{n-1}^n(\theta))= p^n-p^{n-1}+\lambda(\theta).\]
    \item[(3)] For $\theta \in \cO[G_n]$, if $\mu(\pi_{n-1}^n(\theta))=0$, then $\mu(\theta)=0$.
   \end{itemize}    
\end{lemma}
\begin{proof}
    Part (1) is \cite[Lemma 2.6]{PW}; (2) and (3) are discussed in \cite[\S~4]{pollack05} as well as \cite[Lemmas 3.2 and 3.3]{PW}.
\end{proof}

We prove the main theorem of this section that involves an arbitrary sequence of elements in $K[G_n]$, which will be crucial in the proofs of our main results in this article. 
\begin{theorem}
\label{thm:abstract}
Let $\{\theta_n\}_{n\geq0}$ be a sequence in $ K[G_n]$ such that
\begin{itemize}
    \item[(1)]  $\mu(\theta_n)$ is bounded from below as $n \to \infty$, and 
    \item[(2)] there exists $\alpha \in \cO^\times$ such that for all $n\ge1$
\[\theta^\alpha_n := \theta_n - \alpha^{-1} \cor^n_{n-1}(\theta_{n-1}) \in \cO[G_{n}].\]
\end{itemize}
If there exists an integer $n$ such that $\mu(\theta_n)<0$, then
\[\lambda(\theta_n) = p^n-1 \quad\text{ and }\quad \mu(\theta_n) = \mu(\theta_0) < 0\]
 for all $n\geq 0$.
\end{theorem}
\begin{proof}
Define $t:=-\underset{n}{\min} \{\mu(\theta_n)\}>0$, which exists by assumption (1). For each $n$, define $\theta'_n:= \varpi^t\theta_n$, which belongs to $\cO[G_n]$ by the definition of $t$. From (2), we have
\[\theta'_n \equiv \frac{1}{\alpha}\cor^n_{n-1}(\theta'_{n-1})\mod {\varpi^t\cO[G_n]}\]
for all $n\geq1$. Iterating this congruence, we obtain 
\begin{equation}\label{eq: abstract cong for theta'}
    \theta'_n \equiv \frac{1}{\alpha^n} \cor^n_0(\theta'_0) \mod{\varpi^t\cO[G_n]}
\end{equation}
for all $n\geq 0$. 

Let $m$ be an integer such that $\mu(\theta_m)=-t$. Then $\mu(\theta_m')=0$. Setting $n$ to be $m$ in \eqref{eq: abstract cong for theta'}, we deduce from Lemma \ref{lem: corlambda}(2):  $\mu(\cor^m_0(\theta'_0))=\mu(\theta'_0)=0$. Consequently, both sides of the congruence in \eqref{eq: abstract cong for theta'} are non-zero. Hence, from Lemma \ref{lem: corlambda}(2), we deduce the equations
\[\lambda(\theta'_n)=p^n-p^{n-1}+\lambda(\theta'_{n-1})=\cdots=p^n-1+\lambda(\theta'_0)=p^n-1\]
and $\mu(\theta'_n)=\mu(\theta_0')=0$ for all $n\geq0$. Hence,  $\mu(\theta_n)= \mu(\theta_0)= t<0$ and $\lambda(\theta_n)=p^n-1$ for all $n$.  
\end{proof}
\begin{remark}\label{rem: verifying mt elts satisfy abstract assumptions}
    We shall apply Theorem \ref{thm:abstract} to the sequence of $p$-adic Mazur--Tate elements $\{\theta_n(f)\}_{n\geq0}$ attached to a $p$-ordinary Hecke eigenform $f\in S_2(\Gamma_1(N),\epsilon_f)$ with $p\nmid N$. Assumption (1) is satisfied since 
$\mu(\theta_n(f))\geq \ord_\varpi(\Omega_f/\Omega_{f,\bullet})$. Taking $\alpha$ to be the $p$-adic unit root of $x^2-a_p(f)x+\epsilon_f(p)p$ and  $\theta_n^\alpha=\epsilon_f(p)\theta_n(\phi_f^\alpha)$, by \eqref{eq:stabilization} and Lemma~\ref{lem: corlambda}(1), assumption (2) is satisfied due to our choice of normalisation by the period $\Omega_{f,\bullet}$.
    
For an elliptic curve $E$ defined over $\QQ$ that has good ordinary reduction at $p$, the assumptions can be verified similarly for the sequence $\{\theta_n(E)\}_{n\geq0}$: 
    For (1), note that if $r,s\in \mathbb{P}^1(\QQ)$, we have $\phi_E(\{r\}-\{s\})\in \delta_E^{-1}\ZZ$ for an integer $\delta_E$ independent of $r$ and $s$ (see \cite[Lemma 1.2(i)]{MR4669286}). Thus, $\mu(\theta_n(E))\geq -\ord_p(\delta_E)$ for all $n\geq0$. For (2), let $\alpha$ be the $p$-adic unit root of $x^2-a_p(E)x+p$. Then Lemma \ref{mtinteg} says that $\theta_n(\phi_E^\alpha) \in \Zp[G_n]$. Hence, we can take $\theta_n^\alpha=\theta_n(\phi_E^\alpha)$ by \eqref{eq:stabilization} and Lemma~\ref{lem: corlambda}(1).
\end{remark}
\section{Proofs of main results}\label{S: proof of main res}
In this section, we prove Theorems \ref{thmA} and \ref{thmB}.
\subsection{Proof of Theorem \ref{thmA}}\label{S:proofA}
We begin with part (1) of Theorem~\ref{thmA}.
\begin{theorem}\label{thm: Lvaldenom-version2}
    Let $E$ be an elliptic curve defined over $\QQ$ with good ordinary reduction at $p$ such that $\ord_p(L(E,1)/\Omega_{E})<0$. Then we have
    \[\mu(\theta_n(E))=\ord_p(L(E,1)/\Omega_{E})\quad\text{and }\quad\lambda(\theta_n(E))=p^n-1\]
    for all $n\geq 0$.
\end{theorem}
\begin{proof}
If $\ord_p(L(E,1)/\Omega_E)<0$, then $\alpha \equiv 1 \mod{p}$ since
\[L_p(E, T)|_{T=0} =\left(1-\frac{1}{\alpha}\right)^2\frac{L(E,1)}{\Omega_E} \in \Zp\]
by Theorem~\ref{thm: wuthrich Lp integrality}. Thus, $a_p(E)\equiv 1\mod{p}$. 
Consequently, from \eqref{eq:interpolation} it follows that 
$$\mu(\theta_0(E))=\ord_p((a_p-2)L(E,1)/\Omega_E)<0.$$ Therefore, Theorem \ref{thm:abstract} applies to the sequence $\{\theta_n(E)\}_{n\geq0}$ due to the verification carried out in Remark \ref{rem: verifying mt elts satisfy abstract assumptions}, concluding the proof of the theorem.
\end{proof}

Next, we prove the converse to Theorem~\ref{thm: Lvaldenom-version2} under the additional assumption that the $\mu$-invariant of the $p$-adic $L$-function attached to $E$ vanishes, i.e., part (2) of Theorem~\ref{thmA}.
\begin{theorem}\label{thm: converse}
    Let $E$ be an elliptic curve defined over $\QQ$ with good ordinary reduction at an odd prime $p$ with $\mu(L_p(E))=0$. Assume $\lambda(\theta_n(E))=p^n-1$ for all $n$ sufficiently large. Then 
    \[\ord_p\left(\frac{L(E,1)}{\Omega_E}\right)<0.\]
\end{theorem}
\begin{proof}
Since $L_p(E)$ can be realized as the limit of $\theta_n(\phi_E^\alpha)$ as $n\rightarrow\infty$, 
we have $\mu(\theta_n(\phi_E^\alpha))=\mu(L_p(E))=0$ and $\lambda(\theta_n(\phi_E^\alpha)) =\lambda(L_p(E))< p^n-1$ for $n$ large enough.  Recall that 
 \begin{equation}\label{eq: pstabilised MT elts}
       \theta_n(\phi_E^\alpha)=\theta_n(E)-\frac{1}{\alpha}\cor_{n-1}^n(\theta_{n-1}(E)) .
    \end{equation}
Applying Lemma \ref{lem:abstract} gives $0=\mu(\theta_n(\phi_E^\alpha)) > \mu(\theta_n(E))$.  Consequently, Theorem \ref{thm:abstract} applies and we deduce $\mu(\theta_0(E)) < 0$.
By \eqref{eq:interpolation}, this is equivalent to $\ord_p((a_p-2) L(E,1)/\Omega_E) < 0$. As $a_p(E) \in \Zp$, we have  $\ord_p(L(E,1)/\Omega_E) < 0$ as desired.
\end{proof}

In a completely analogous manner, we obtain the following extension of Theorems \ref{thm: Lvaldenom-version2} and \ref{thm: converse} to modular forms of weight 2.  Recall that the period $\Omega_{f,\bullet}$ is chosen so that $L_p(f)$ is integral.
\begin{theorem}\label{thm: ThA for mod forms}
    Let $f \in S_2(\Gamma_1(N),\epsilon_f)$ be a Hecke eigenform and $p$ an odd prime such that $p\nmid N$ and $f$ is ordinary at $p$. Let $\varpi$ denote a uniformizer for $K_f$, the completion of the field of Fourier coefficients of $f$ at a prime above $p$.
    \begin{itemize}[leftmargin=0.5cm] 
    \item[(1)] If $\ord_\varpi(L(f,1)/\Omega_{f,\bullet})<0$, then
    \[\mu(\theta_n(f))=\ord_\varpi(L(f,1)/\Omega_{f,\bullet})\quad\text{and }\quad\lambda(\theta_n(f))=p^n-1\]
    for all $n\geq0$.
    \item[(2)] If $\mu(L_p(f))=0$ and $\lambda(\theta_n(f))=p^n-1$
    for $n\gg0$, then
    \[\ord_\varpi\left(\frac{L(f,1)}{\Omega_{f,\bullet}}\right)<0.\]
\end{itemize}
\end{theorem}

\subsection{Proof of Theorem \ref{thmB}}\label{S:proofB}
In order to prove Theorem~\ref{thmB}, we discuss how the integrality of Mazur--Tate elements influences the corresponding Iwasawa invariants  at good ordinary primes. 
\begin{theorem}\label{thm: mt invariants are Lp invariants}
    Let $E$ be an elliptic curve over $\QQ$ with good ordinary reduction at an odd prime $p$ with $\mu(L_p(E))=0$. Suppose that $\theta_n(E) \in \Lambda_n$ for all $n\gg 0$. Then 
    \[\mu(\theta_n(E))=0 \quad\text{and}\quad \lambda(\theta_n(E))=\lambda(L_p(E))\]
    for all $n$ sufficiently large.
\end{theorem}
\begin{proof}
    This proof is based on the argument used in \cite[proof of Proposition 3.7]{PW}. Note that $\theta_{n-1}(E) \in \Lambda_{n-1}$ implies $\cor_{n-1}^{n}(\theta_{n-1}(E)) \in \Lambda_n$. Further, $\mu(\theta_n(\phi_E^\alpha))=\mu(L_p(E))=0$ and $\lambda(\theta_n(\phi_E^\alpha))=\lambda(L_p(E))$ for all $n\gg0$ since $\theta_n(\phi_E^\alpha)$ converges to $L_p(E)$ as $n\rightarrow \infty$. 

    Assume that $n$ is chosen so that $\mu(\theta_n(\phi_E^\alpha))=0$.
As the right-hand side of \eqref{eq: pstabilised MT elts} consists of two elements in $\Lambda_n$, one of them must have zero $\mu$-invariant. Moreover, $\mu(\cor_{n-1}^{n}(\theta_{n-1}(E)))=\mu(\theta_{n-1}(E))$ by Lemma \ref{lem: corlambda}(2). Hence, we deduce that $\mu(\theta_{n'}(E))=0$ for some $n'$. Further, $n'$ can be arbitrarily large. Suppose that $n'$ is chosen so that $\mu(\theta_{n'}(E))=0$. If $\mu(\theta_{n'+1}(E))>0$, this would imply that
    \[\theta_{n'+1}(\phi_E^\alpha)\equiv -\frac{1}{\alpha}\cor_{n'}^{n'+1}(\theta_{n'}(E)) \mod{p\Lambda_{n'+1}}.\]
    Since both sides of this congruence have trivial $\mu$-invariants, Lemma \ref{lem: corlambda}(2) gives  
    \[\lambda(\theta_{n'+1}(\phi_E^\alpha))=p^{n'+1}-p^{n'}+ \lambda(\theta_{n'}(E)),\]
    which cannot hold when $n'$ is sufficiently large  since $\lambda(\theta_{n'+1}(\phi_E^\alpha))=\lambda(L_p(E))$ is bounded. So, we must have $\mu(\theta_{n'+1}(E))=0$. In particular, $\mu(\theta_{n}(E))=0$ for $n\gg0$. In addition, 
    \[
\lambda(L_p(E, T))=    \lambda(\theta_n(\phi^\alpha_E))<\lambda(\cor_{n-1}^n(\theta_{n-1})= p^n-p^{n-1}+\lambda(\theta_{n-1}(E))
    \]
    for $n\gg0$. Hence, from Equation \eqref{eq: pstabilised MT elts}, we deduce
    \[\lambda(\theta_n(E)) = \min\left\{\lambda(\theta_n(\phi^\alpha_E)),\ \lambda(\cor_{n-1}^n(\theta_{n-1}(E)))\right\}=\lambda(\theta_n(\phi^\alpha_E))=\lambda(L_p(E.T)),\]
    as desired.
\end{proof}
\begin{remark}\label{rk: remark on maximality for modforms}
Let $f \in S_2(\Gamma_1(N),\epsilon_f)$ be a $p$-ordinary Hecke eigenform with $p\nmid N$. Let $\alpha$ be the $p$-adic unit root of $x^2-a_p(f)+\epsilon_f(p)p$ and let $f_\alpha$ denote the ordinary $p$-stabilisation of $f$. Let $\varpi$ denote a uniformizer of $K_f$ and $\cO_f$ denote the ring of integers of $K_f$. Recall that we normalise $\theta_n(f)$ by the periods $\Omega_{f,\bullet}^\pm$, which ensures the integrality of $L_p(f)$. If $\mu(L_p(f))=0$ and $\theta_n(f) \in \cO_f[G_n]$, then 
\[\mu(\theta_n(f))=0 \quad\text{and}\quad \lambda(\theta_n(f))=\lambda(L_p(f))\]
for all $n$ sufficiently large. This can be shown by an argument similar to the one used in the proof of Theorem \ref{thm: mt invariants are Lp invariants}.

In this article, we focus on weight 2 forms. Mazur--Tate elements can be defined for cusp forms of higher weight (see \cite[\S~2]{PW}) and if the  modular symbols are normalised by the cohomological periods, one can obtain the same result as in Theorem \ref{thm: mt invariants are Lp invariants}. More precisely, if the $\mu$-invariant of the associated $p$-adic $L$-function vanishes, the $\lambda$-invariant of the Mazur--Tate element always matches the $\lambda$-invariant of the $p$-adic $L$-function for $n\gg0$.
\end{remark}

We now give the proof of Theorem~\ref{thmB}.
\begin{theorem}\label{thm: classify}
    Let $E$ be an elliptic curve over $\QQ$ with good ordinary reduction at an odd prime $p$. Assume $E$ is isogenous over $\QQ$ to an elliptic curve $E'/\QQ$ with $\mu(L_p(E'))=0$. Then, we have the following dichotomy:

\begin{itemize}
        \item[(a)] $\theta_n(E') \in \Lambda_n$ for all $n \geq 0$.  In this case, $$\mu(\theta_n(E))=\mu(L_p(E)) \quad \text{and} \quad \lambda(\theta_n(E))=\lambda(L_p(E))
        $$
        for $n$ sufficiently large, or
        \item[(b)] $\theta_n(E') \notin \Lambda_n$ for all $n \geq 0$.  In this case,
        $$
        \mu(\theta_n(E))=\ord_p(L(E,1)/\Omega_E) \quad \text{and} \quad \lambda(\theta_n(E))=p^n-1
        $$
        for all $n \geq 0$.
    \end{itemize}
       Furthermore, case (b) occurs if and only if $\ord_p(L(E',1)/\Omega_{E'})<0$.
\end{theorem}
\begin{proof}
Since $\theta_n(E)=\theta_n(E')\cdot\Omega_{E'}/\Omega_E$ and  $L_p(E,T)=L_p(E',T)\cdot\Omega_{E'}/\Omega_E$, we have
\[
\lambda(\theta_n(E))=\lambda(\theta_n(E'))\quad \text{ and }\quad\mu(\theta_n(E))-\mu(L_p(E,T))=\mu(\theta_n(E'))-\mu(L_p(E',T)).\] Thus, we may assume $E=E'$.

Suppose $\theta_{n}(E) \notin \Lambda_{n}$ for some $n$. In other words,   $\mu(\theta_n(E))<0$. Theorem \ref{thm:abstract} combined with Remark~\ref{rem: verifying mt elts satisfy abstract assumptions} implies that the conditions described in case (b) hold. In addition, Theorem~\ref{thm: converse} tells us that $\ord_p(L(E,1)/\Omega_E)<0$ in this case.

Otherwise, suppose that $\theta_n(E)\in \Lambda_n$ for all $n\ge0$. In this case, Theorem \ref{thm: mt invariants are Lp invariants} implies that the conditions described in case (a) hold. Furthermore, the contrapositive of Theorem~\ref{thm: Lvaldenom-version2} tells us that $\ord_p(L(E,1)/\Omega_E)\ge0$.
\end{proof}
\begin{remark}
Assume that  $\mu(L_p(E))=0$.
    Theorem~\ref{thm: classify} implies that $\theta_n(E) \in \Lambda_n$ for all good ordinary primes $p>2$ except when $\ord_p(L(E,1)/\Omega_E)<0$. If $p>7$, then Mazur's theorem on torsion groups of elliptic curves over $\QQ$, which says that $E(\QQ)[p]=0$. In particular, the Birch--Swinnerton-Dyer conjecture predicts $\ord_p(L(E,1)/\Omega_E)\ge0$. Therefore, we expect that only case (a) of Theorem~\ref{thm: classify} occurs when $p>7$.
\end{remark}

One can obtain the analogue of Theorem \ref{thm: classify} for a $p$-ordinary Hecke eigenform $f \in S_2(\Gamma_1(N),\epsilon_f)$ with $p\nmid N$ in the same way, combining the previous argument with Remark \ref{rk: remark on maximality for modforms}: 
\begin{theorem}\label{thm: ThB for mod forms}
Let $f \in S_2(\Gamma_1(N),\epsilon_f)$ be a $p$-ordinary Hecke eigenform with $p\nmid N$. 
Suppose $\Omega_{f,\bullet}$ is chosen so that $\mu(L_p(f))=0$. We have the following dichotomy:
\begin{itemize}
        \item[(a)] $\theta_n(f) \in \cO_f[G_n]$ for all $n \geq 0$.  In this case, 
        $$\mu(\theta_n(f))=\mu(L_p(f))=0 \quad \text{and} \quad \lambda(\theta_n(f))=\lambda(L_p(f))
        $$
        for $n$ sufficiently large, or
        \item[(b)] $\theta_n(f) \notin \cO_f[G_n]$ for all $n\geq0$.  In this case, 
        $$\mu(\theta_n(f))=\ord_p(L(f,1)/\Omega_{f,\bullet})<0 \quad \text{and} \quad \lambda(\theta_n(f))=p^n-1.
        $$
    \end{itemize}
\end{theorem}

\section{Examples and remarks on congruences with boundary symbols}\label{sec: remarks, examples}
In this section, we relate the maximality of $\lambda$-invariants of Mazur--Tate elements to certain congruences satisfied by the corresponding modular symbol. We begin with the following proposition, which is a reformulation of Theorem \ref{thm:abstract} in terms of a congruence condition on modular symbols at divisors of the form $\{\infty\}-\{a/p^n\}$ for $a\in (\ZZ/p^n\ZZ)^\times$ and $n\geq0$.
\begin{proposition}\label{lem: maxl}
For a positive integer $N$ and an odd prime $p$,  let $\phi \in \Symb(\Gamma_0(N),\Zp)$. Let $\ord_p(\phi(\{\infty\}-\{0\}))=m$.
If there exists a constant $\alpha\in \Zp^\times$ and an integer $t>m$ such that
\begin{equation}\label{eq:cong}
    \phi(\{\infty\}-\{{a}/{p^{n+1}}\}\equiv \alpha\phi(\{\infty\}-\{{a}/{p^{n}}\})  \mod{p^{t}}
\end{equation} 
 for all $a \in (\ZZ/p^n\ZZ)^\times$ and $n\geq 0$, then 
  \[\mu(\theta_n(\phi))=m\quad \text{ and }\quad \lambda(\theta_n(\phi))=p^n-1\]
  for all $n\geq 0$.
\end{proposition}
\begin{proof}
    The condition \eqref{eq:cong} is equivalent to 
    \[\phi(\{\infty\}-\{{a}/{p^n}\}) \equiv \alpha \phi\mid\begin{psmallmatrix}
        p & 0 \\ 0 & 1
    \end{psmallmatrix}(\{\infty\}-\{{a}/{p^n}\}) \mod{p^t}\] 
    for all $a \in (\ZZ/p^n\ZZ)^\times$ and $n\geq 1$.
    Since $\theta_n(\phi)$ is determined by the values of $\phi$ on divisors of the form $\{\infty\}-\{a/p^n\}$, we apply Lemma~\ref{lem: corlambda}(1) to deduce
    \[\theta_n(\phi)\equiv \alpha\cdot\cor_{n-1}^{n}(\theta_{n-1}(\phi))\equiv \cdots \equiv \alpha^n\cdot\cor^n_0(\theta_0(\phi)) \mod{p^t}.\]
    But 
    $$
    \theta_0(\phi)= \sum_{a=1}^{p-1} \phi(\{\infty\}-\{a/p\}) \sigma_1 \overset{\eqref{eq:cong}}{\equiv} \alpha \sum_{a=1}^{p-1} \phi(\{\infty\}-\{a\}) \sigma_1 =\alpha (p-1)\cdot\phi(\{\infty\}-\{0\})\cdot \sigma_1 \not \equiv 0 \mod{p^{t}}
    $$
    by hypothesis. The proposition now follows from Lemma~\ref{lem: corlambda}(2).
\end{proof}
Ihara's lemma, in the form described in \cite[Theorem 3.5]{PW},  for a Hecke eigensymbol $\varphi_f \in \Symb(\Gamma_0(N),\Zp)$ (corresponding to a weight 2 eigenform $f$) implies the following:\
if, for some constant $c_0\in \Zp^\times$,
\begin{equation}\label{eq: ihara}
    \varphi_f\equiv c_0\varphi_f|\begin{psmallmatrix}
    p & 0 \\ 0 & 1
\end{psmallmatrix}\mod{p}
\end{equation}
then $f$ is congruent modulo $p$ to an Eisenstein series of weight $2$. 
Note that if \eqref{eq: ihara} holds and $\varphi_f(\{\infty\}-\{0\})\in \Zp^\times$, Proposition \ref{lem: maxl} implies that 
\[\lambda(\theta_n(f))=p^n-1 \text{ for all } n\geq0,\]
where the modular symbols are normalised by the cohomological periods $\Omega_f^\pm$ associated with $f$.

\subsection{Congruences with boundary symbols}\label{sec:boundary}
We give a brief review on boundary symbols, following \cite{bel-das}. Let  $R$ be a commutative ring. 
We have the exact sequence 
    \begin{equation}\label{eqn:longcohom}
        0 \xrightarrow{} R^\Gamma \xrightarrow{} \text{Hom}_{\Gamma}(\dDelta,R) \xrightarrow{b} \Symb(\Gamma, R) \xrightarrow{h} {H}^1(\Gamma,R).
    \end{equation}
\begin{defn}\label{defn: bsym}
    The map $b$ in \eqref{eqn:longcohom} is called the \textbf{boundary map} and its image, denoted by $\BSymb(\Gamma, R)$, is called the module of \textbf{boundary modular symbols} (or simply \textbf{boundary symbols}). We denote the space of weight $2$ boundary symbols by $\BSymb(\Gamma, R)$.
\end{defn}
The exact sequence (\ref{eqn:longcohom}) yields an isomorphism of Hecke-modules $$\text{BSymb}(\Gamma, R) \cong \text{Hom}_{\Gamma} (\dDelta, R)/ R^\Gamma,$$
relating modular symbols to boundary symbols.
Furthermore, there is a short exact sequence
$$0 \to \text{BSymb}_\Gamma(R) \to \Symb(\Gamma,R) \to H^1(\Gamma, R).$$
This space of boundary symbols can be identified with the space of weight $2$ Eisenstein series under the Eichler--Shimura isomorphism (see Remark \ref{rk:boundary} below and note that a notion of modular symbols that is dual to the one discussed here is utilized therein). These symbols can be considered as $\Gamma$-invariant maps on the set of divisors $\dDelta$, greatly simplifying the computation.
\begin{remark}\label{rk:boundary}(\textbf{On the Eichler--Shimura isomorphism})
The space of (weight 2) modular symbols $\Symb(\Gamma,\CC)=\Hom_{\Gamma}(\dDelta, \CC)$ can be identified with the direct sum $S_2(\Gamma)\oplus S^{\mathrm{anti}}_2(\Gamma)\oplus E_2(\Gamma)$, where $S_2(\Gamma),S^{\mathrm{anti}}_2(\Gamma)$ and $E_2(\Gamma)$ denote the $\CC$-vector space of holomorphic cusp forms, anti-holomorphic cusp forms and Eisenstein series, respectively. This is the \emph{Eichler--Shimura isomorphism}. In terms of the cohomology of the modular curve for $\Gamma=\Gamma_0(N)$, we have the following exact sequence of Hecke-modules  
    \[0\to \BSymb(\Gamma_0(N),\CC) \to H^1_c(Y_0(N),\CC) \xrightarrow[]{} H^1(Y_0(N),\CC)\to E_{2}(\Gamma)\to 0 \]
    (see \cite[Proposition 2.5]{bel-das}). Recall that there is a canonical isomorphism $\Symb(\Gamma_0(N), \CC) \cong H^1_c(Y_0(N),\CC)$. Thus, $\BSymb(\Gamma,\CC)$ is the summand corresponding to $E_2(\Gamma)$ inside $\Symb(\Gamma,\CC)$ and can be identified with the image of $\Hom(\dDelta,\CC)$ inside $\Hom_{\Gamma}(\dDelta^0, \CC)$ (via the map $\dDelta\to\dDelta^0$ given by restricting to degree-zero divisors).
\end{remark}

Let $\phi_{E,\mathrm{Coh}}$ denote the modular symbol attached to an elliptic curve $E$ normalised by the cohomological periods of $E$ for a given prime $p$, i.e.,
\[\phi_{E,\mathrm{Coh}} \colonequals \frac{\xi^+_{f_E}}{\Omega^+_{f_E}}+  \frac{\xi^-_{f_E}}{\Omega^-_{f_E}}.\]
We prove a straightforward consequence of the results from previous sections. 
\begin{proposition}\label{thm: bsym to Lval}
    Let $E$ be an elliptic curve over $\QQ$ with good ordinary reduction at an odd prime $p$ and conductor $N_E$. Assume $L(E,1)/\Omega_{f_E}\in \Zp^\times$. Suppose $\phi_{E,\mathrm{Coh}}^+$ is congruent modulo $p$ to a boundary symbol $\phi_B \in \BSymb(\Gamma_0(N_E),\Zp)$ of level $\Gamma_0(N_E)$. Then 
    \[\lambda(\theta_n(E))=p^n-1\  \forall \ n\geq 0.\] 
    Furthermore, if $\mu(L_p(E))=0$, we have
    \[\ord_p(L(E,1)/\Omega_E)<0.\]
\end{proposition}

\begin{proof}
 The congruence between $\phi_{E,\mathrm{Coh}}^+$ with a weight 2 boundary symbol and the fact that divisors of the form $\{a/p^n\}$ and $\{a/p^{n-1}\}$ are $\Gamma_0(N_E)$-equivalent imply the congruence 
    \[\phi_{E,\mathrm{Coh}}^{+} (\{\infty\}-\{a/p^n\}) \equiv \phi_{E,\mathrm{Coh}}^{+} (\{\infty\}-\{a/p^{n-1}\})\mod p\] for all $n\geq 1$. The first assertion of the proposition follows from Proposition \ref{lem: maxl}, whereas the second assertion is a consequence of Theorem \ref{thm: converse}. 
\end{proof}

This illustrates one way in which case (b) of Theorem \ref{thm: classify} occurs. Interestingly, all the examples generated via our computations having $\lambda(\theta_n(E))=p^n-1$ for a good ordinary prime $p>2$ satisfy the hypotheses of Proposition \ref{thm: bsym to Lval}.
\subsection{Examples}\label{ssec: examples}
\begin{example}\label{example1}
    Let $E=$\href{https://www.lmfdb.org/EllipticCurve/Q/26b1/}{26b1} and $p=7$. It has a non-trivial torsion point of order $7$ over $\QQ$. From LMFDB, we see that $L(E,1)/\Omega_E=1/7$. It follows from Theorem \ref{thm: Lvaldenom-version2} that
    \[\lambda(\theta_n(E))=p^n-1\] for $n \geq 0$. 
    Furthermore, we see that $\phi_{E,\mathrm{Coh}}^+$ is congruent modulo $p$ to a boundary symbol (taking values in $\Fp$) defined as 
    \[\psi(\{r\})=\begin{cases}
        -1/2 & {\text{ if } r \in \Gamma_0(26)\{0\}\cup \Gamma_0(26)\{1/13\}}\\
        1/2 & {\text{ if } r \in \Gamma_0(26)\{\infty\}\cup \Gamma_0(26)\{1/2\}}
    \end{cases}\]
    for $r \in \QQ$. Here $\Gamma_0(26)\{c\}$ denotes the equivalence class of the cusp $c$. 
\end{example}    
We have found examples of elliptic curves with additive reduction at $p$ where $\lambda(\theta_n(E))$ is maximal. Interestingly, in all of these examples, we observe that $\ord_p(L(E,1)/\Omega_E)<0$. 
\begin{example}\label{example: 50b}
    Let $E=$\href{https://www.lmfdb.org/EllipticCurve/Q/50b1/}{50b1} and $p=5$. This curve admits a non-trivial $5$-torsion over $\QQ$ and $\phi_{E,\mathrm{Coh}}(\{\infty\}-\{0\})=1$. We have $L(E,1)/\Omega_E=1/5$.
    We observe, for $n\geq 0$, 
    \[\lambda(\theta_{n}(E))=p^n-1.\]
    Moreover, $\phi_{E,\mathrm{Coh}}^+$ is indeed congruent to a weight 2 boundary symbol on $\Gamma_0(50)$ modulo $5$.
\end{example}
\begin{remark}
  Note that when $E$ has $p$-torsion over $\QQ$, the mod $p$ representation is reducible and one can exhibit an Eisenstein series of weight $2$ having Hecke eigenvalues congruent modulo $p$ to those of $E$. In all the examples that exhibit the maximality of $\lambda(\theta_n(E))$ with $E$ semistable and $p$ a prime of good ordinary reduction that we have found, the hypotheses of \cite[Theorem 1.5.1]{WAKE2021107543} can be verified. So, it follows that, in each of these examples, the relevant Eisenstein completion of the Hecke algebra is a Gorenstein ring. 
  We expect that this implies that the space of $\Fp$-valued modular symbols with Hecke eigenvalues congruent to those of $E$ is one dimensional, implying that there is a congruence with a boundary symbol. 
\end{remark}
We end with another example to demonstrate a consequence of Proposition \ref{thm: bsym to Lval} in this direction.
\begin{example} \label{example: 174b}
    Let $E=$\href{https://www.lmfdb.org/EllipticCurve/Q/174b1/}{174b1} and $p=7$. This curve has $\mu(L_p(E))=0$, and $L(E,1)/\Omega_{f_E}\in \Zp^\times$. The curve $E$ admits a non-trivial $7$-torsion point over $\QQ$, which implies that $E$ is congruent to a weight 2 Eisenstein series of level $\Gamma_0(174)$ modulo $7$. However, we have $L(E,1)/\Omega_E=1$, so Proposition~\ref{thm: bsym to Lval} tells us that $\phi_E$ cannot be congruent to a weight 2 boundary symbol on $\Gamma_0(174)$ modulo $7$. This implies that mod $p$ multiplicity one fails in this case. 
\end{example}
\bibliographystyle{amsalpha}
\bibliography{references}
\end{document}